\documentclass{article}
\usepackage{amsfonts}
\usepackage{amsmath}

\providecommand{\U}[1]{\protect \rule{.1in}{.1in}}
\newtheorem{theorem}{Theorem}

\newtheorem{corollary}[theorem]{Corollary}

\newtheorem{lemma}[theorem]{Lemma}

\newenvironment{proof}[1][Proof]{\noindent \textbf{#1.} }{\  \rule{0.5em}{0.5em}}
\include {mak}
\begin{document}

\begin{center}
{\LARGE A remarkable congruence involving Appell polynomials}

\  \ 

{\Large Abdelkader Benyattou } \  \\[0pt]
\  \  \\[0pt]
UNIV-DJELFA, Dep. of Mathematics and Informatics, RECITS Laboratory, Djelfa,
Algeria.

a.benyattou@univ-djelfa.dz, abdelkaderbenyattou@gmail.com \  \  \ 

\  \  \  \  \ 

\textbf{Abstract}
\end{center}

\begin{quote}
In this paper, by the umbral calculus method, we give a remarkable
congruence involving Appell polynomials. Some applications on derangement
polynomials are also presented.
\end{quote}

\noindent \textbf{Keywords.} Congruence, Appell polynomials, classical
umbral calculus.

\noindent \textbf{2010 MSC.} 11B73, 11A07, 05A40.

\section{Introduction}
Let $\left( \mathcal{A}_{n}\left( x\right) \right) $ be a sequence of Appell
polynomials with integer coefficients defined by \cite{Ap}: 
\begin{equation*}
\sum_{n\geq 0}\mathcal{A}_{n}\left( x\right) \frac{t^{n}}{n!}=F\left(
t\right) \exp \left( xt\right) ,\text{ \ }%
\mathcal{A}_{n}\left( 0\right)=\mathcal{A}_{n} \in \mathbb{N},  \label{h}
\end{equation*}%
where \begin{equation*}
F\left( t\right)=1+\sum _{{n\geq 1}} \mathcal{A}_{n} \frac{t^{n}}{n!},
\end{equation*} Let $\mathbf{A}$ be the Appell umbra defined
by $\mathbf{A}^{n}=\mathcal{A}_{n},$ then we can define the generalized Appell umbra  $\mathbf{A}_{\mathbf{x}}$ as follows   \begin{equation*}
\mathbf{A}_{\mathbf{x}}^{n}=\mathcal{A}_{n}\left( x\right)=\left(\mathbf{A}+x\right) ^{n} 
\end{equation*}
see \cite{Mihoubi}. 
For more information on the umbral calculus see \cite{Ben,ges,rota,rom}%
. In the remainder of this paper, for any polynomials $f$ and $g$, we denote
by $f(x)\equiv g(x)$ to mean $f(x)\equiv g(x)$ (mod $\mathbb{Z}[x]$)
and for any numbers $a$ and $b$ by $a\equiv b$ we mean $a\equiv b$ (mod $p$).

\section{Congruences involving Appell polynomials }
Main result is given by the following theorem.
\begin{theorem}
\label{T} Let $f$ be a polynomial in $\mathbb{Z}\left[ x\right] ,$ $m\geq
0,s\geq 1$ be integers and $p$ be an odd prime number. Then if for any
integer $n\geq 0,$ there exists an integer $t$ such that $\mathcal{A}%
_{n+p} \equiv t\mathcal{A}_{n},$ we have the following congruence 
\begin{equation*}
\left( \mathbf{A}+x\right) ^{mp^{s}}f\left( \mathbf{A}+x\right) \equiv
\left( x^{p^{s}}+t\right) ^{m}f\left( \mathbf{A}+x\right) .  \label{i}
\end{equation*}%
This congruence is equivalent when $f\left( x\right) =x^{n}$ to 
\begin{equation*}
\mathcal{A}_{n+mp^{s}}\left( x\right) \equiv \left( x^{p^{s}}+t\right) ^{m}%
\mathcal{A}_{n}\left( x\right) .  \label{j}
\end{equation*}%
In particular for $n=0$, we have%
\begin{equation*}
\mathcal{A}_{mp^{s}}\left( x\right) \equiv \left( x^{p^{s}}+t\right) ^{m}.
\label{k}
\end{equation*}
\end{theorem}
\begin{proof}
It suffices to take $f\left( x\right) =x^{n}.$ For $m=1$ we proceed by
induction on $s.$ Indeed, for $s=1$ we have 
\begin{align*}
\left( \mathbf{A}_{\mathbf{x}}^{p}-t\right) \mathbf{A}_{\mathbf{x}}^{n}& =%
\mathbf{A}_{\mathbf{x}}^{n+p}-t\mathbf{A}_{\mathbf{x}}^{n} \\
& =\left( \mathbf{A}+x\right) ^{p}\left( \mathbf{A}+x\right) ^{n}-t\mathbf{A}%
_{\mathbf{x}}^{n} \\
& \equiv \left( \mathbf{A}^{p}+x^{p}\right) \left( \mathbf{A}+x\right) ^{n}-t%
\mathbf{A}_{\mathbf{x}}^{n} \\
& =x^{p}\mathbf{A}_{\mathbf{x}}^{n}+\sum_{k=0}^{n}{\binom{n}{k}}\mathbf{A}%
^{n+p-k}x^{k}-t\mathbf{A}_{\mathbf{x}}^{n} \\
& \equiv x^{p}\mathbf{A}_{\mathbf{x}}^{n}+t\sum_{k=0}^{n}{\binom{n}{k}}%
\mathbf{A}^{n-k}x^{k}-t\mathbf{A}_{\mathbf{x}}^{n} \\
& =x^{p}\mathbf{A}_{\mathbf{x}}^{n}.
\end{align*}%
Assume it is true for $s\geq 1.$ Then we have 
\begin{align*}
\mathbf{A}_{\mathbf{x}}^{n}\left( \mathbf{A}_{\mathbf{x}}^{p^{s+1}}-t\right) &
=\mathbf{A}_{\mathbf{x}}^{n}\left( \left( \mathbf{A}_{\mathbf{x}%
}^{p^{s}}-t+t\right) ^{p}-t\right) \\
& =\mathbf{A}_{\mathbf{x}}^{n}\left( \left( \mathbf{A}_{\mathbf{x}%
}^{p^{s}}-t\right) ^{p}+t-t\right) \\
& \equiv \mathbf{A}_{\mathbf{x}}^{n}\left( \mathbf{A}_{\mathbf{x}%
}^{p^{s}}-t\right) ^{p} \\
& =\left[ \mathbf{A}_{\mathbf{x}}^{n}\left( \mathbf{A}_{\mathbf{x}%
}^{p^{s}}-t\right) \right] \left( \mathbf{A}_{\mathbf{x}}^{p^{s}}-t\right)
^{p-1} \\
& \equiv x^{p^{s}}\mathbf{A}_{\mathbf{x}}^{n}\left( \mathbf{A}_{\mathbf{x}%
}^{p^{s}}-t\right) ^{p-1} \\
& =x^{p^{s}}\left[ \mathbf{A}_{\mathbf{x}}^{n}\left( \mathbf{A}_{\mathbf{x}%
}^{p^{s}}-t\right) \right] \left( \mathbf{A}_{\mathbf{x}}^{p^{s}}-t\right)
^{p-2} \\
& \equiv x^{2p^{s}}\mathbf{A}_{\mathbf{x}}^{n}\left( \mathbf{A}_{\mathbf{x}%
}^{p^{s}}-t\right) ^{p-2} \\
& \  \  \vdots \\
& \equiv \left( x^{p^{s}}\right) ^{p}\mathbf{A}_{\mathbf{x}}^{n} \\
& =x^{p^{s+1}}\mathbf{A}_{\mathbf{x}}^{n}.
\end{align*}%
So, we proved that $\left( \mathbf{A}+x\right) ^{p^{s}}f\left( \mathbf{A}%
+x\right) \equiv \left( x^{p^{s}}+t\right) f\left( \mathbf{A}+x\right) .$
For $m\geq 1$ we use this last congruence to obtain 
\begin{align*}
\left( \mathbf{A}+x\right) ^{mp^{s}}f\left( \mathbf{A}+x\right) & =\left( 
\mathbf{A}+x\right) ^{p^{s}}\left( \mathbf{A}+x\right) ^{\left( m-1\right)
p^{s}}f\left( \mathbf{A}+x\right) \\
& \equiv \left( x^{p^{s}}+t\right) \left( \mathbf{A}+x\right) ^{\left(
m-1\right) p^{s}}f\left( \mathbf{A}+x\right) \\
& \  \  \vdots \\
& \equiv \left( x^{p^{s}}+t\right) ^{m}f\left( \mathbf{A}+x\right) .
\end{align*}
\end{proof}
\begin{corollary}
Let $f$ be a polynomial in $\mathbb{Z}\left[ x\right] ,$ $m\geq 0,$ $s\geq 1$
be integers and $p$ be an odd prime number. Then if for any integer $n\geq
0, $ there exists an integer $t$ such that $\mathcal{A}_{n+p} \equiv t\mathcal{A}_{n}$ $,$ there holds%
\begin{equation*}
\left( \mathbf{A}+x\right) ^{m_{1}p+\cdots +m_{s}p^{s}}f\left( \mathbf{A}%
+x\right) \equiv \left( x^{p}+t\right) ^{m_{1}}\cdots \left(
x^{p^{s}}+t\right) ^{m_{s}}f\left( \mathbf{A}+x\right) .  \label{l}
\end{equation*}%
This congruence is equivalent when $f\left( x\right) =x^{m_{0}}$ to 
\begin{equation*}
\mathcal{A}_{m_{0}+{m_{1}p+\cdots +m_{s}p^{s}}}\left( x\right) \equiv \left(
x^{p}+t\right) ^{m_{1}}\cdots \left( x^{p^{s}}+t\right) ^{m_{s}}\mathcal{A}%
_{m_{0}}\left( x\right) .  \label{m}
\end{equation*}%
In particular, for $m_{0}=0,$ we have%
\begin{equation*}
\mathcal{A}_{{m_{1}p+\cdots +m_{s}p^{s}}}\left( x\right) \equiv \left(
x^{p}+t\right) ^{m_{1}}\cdots \left( x^{p^{s}}+t\right) ^{m_{s}}.  \label{n}
\end{equation*}
\end{corollary}
\begin{proof}
By the Theorem \ref{T} we can write 
\begin{align*}
\left( \mathbf{A}+x\right) ^{m_{1}p+\cdots +m_{s}p^{s}}f\left( \mathbf{A}%
+x\right) & =\left( \mathbf{A}+x\right) ^{m_{1}p}\left( \mathbf{A}+x\right)
^{m_{2}p^{2}+\cdots +m_{s}p^{s}}f\left( \mathbf{A}+x\right) \\
& \equiv \left( x^{p}+t\right) ^{m_{1}}\left( \mathbf{A}+x\right)
^{m_{2}p^{2}+\cdots +m_{s}p^{s}}f\left( \mathbf{A}+x\right) \\
& \  \  \vdots \\
& \equiv \left( x^{p}+t\right) ^{m_{1}}\cdots \left( x^{p^{s}}+t\right)
^{m_{s}}f\left( \mathbf{A}+x\right) .
\end{align*}
\end{proof}
\section{Applications }
In this section we give some applications of Theorem \ref{T}. \\
The derangement polynomials are defined by%
\begin{equation*}
\mathcal{D}_{n}\left( x\right) =\sum_{k=0}^{n}{\binom{n}{k}}\mathcal{D}%
_{n-k}x^{k}=\sum_{k=0}^{n}{\binom{n}{k}}k!\left( x-1\right) ^{n-k},
\label{o}
\end{equation*}%
where $\mathcal{D}_{n}=\mathcal{D}_{n}\left( 0\right) $ is the $n$-th derangement number, denoted by $\mathcal{D}_{n}$ counting the number of
permutation of the set $\left[ n\right] $ without a fixed point. With exponential generating series $$\sum_{i=0}^{\infty} \mathcal{D}_{n} \left( x\right)\frac{t^{n}}{n!} =\frac{e^{-t}}{1-t}  e^{xt}  .$$  Recall that for any integer $n\geq 2,$ we have%
\begin{equation} \label{e}
\mathcal{D}_{n}\left( x\right) =n\mathcal{D}_{n-1}\left( x\right) +\left(x
-1\right) ^{n}
\end{equation}
 The derangement numbers satisfy the following congruence \cite{sun2}. For any non-negative integer $n$ and any prime number $p,$   $$\mathcal{D}_{n+p}\equiv -\mathcal{D}_{n}.$$ \\ Let $D$ be the derangement umbra defined by $D^{n}=\mathcal{D}_{n}$ and   $\mathbf{D}_{\mathbf{x}}$ be the generalized derangement umbra defined by $\mathbf{D}_{\mathbf{x}}^{n}=\left( D+x\right) ^{n}=\mathcal{D}_{n}\left( x\right).$  
\begin{lemma}  \label{L}
Let $g\left( x\right) =\sum_{k=2}^{n} a_{k}x^{k},$ be a polynomial with real coefficients. Then $$g\left(\mathbf{D}_{\mathbf{x}}\right) = g'\left(\mathbf{D}_{\mathbf{x}}\right)+\sum_{k=2}^{n} a_{k}\left( x-1\right) ^{k}.$$  where $g^{\prime }$ is the derivative polynomial.
\end{lemma}
\begin{proof}
We have
\begin{align*}
g\left(\mathbf{D}_{\mathbf{x}}\right) &=\sum_{k=2}^{n} a_{k}\left(\mathbf{D}_{\mathbf{x}}\right)^{k} \\ & =\sum_{k=2}^{n} a_{k}\left( k\mathbf{D}_{\mathbf{x}}^{k-1}+\left( x-1\right) ^{k}\right)\\ =& \sum_{k=2}^{n} a_{k} k\mathbf{D}_{\mathbf{x}}^{k-1}+\sum_{k=2}^{n} a_{k}\left( x-1\right) ^{k}.
\end{align*}
\end{proof} When we replace $ g\left( x\right)  $ in lemma \ref{L},by $ x^{n} $ we obtain the identity (\ref{e})
\begin{corollary}
\label{C1} For any integers $n\geq 1,$ $s\geq 1,$ $m\geq 0$ and for any
prime number $p\geq 3,$ there holds%
\begin{equation}
\mathcal{D}_{n+mp^{s}}\left( x\right) \equiv \left( x^{p^{s}}-1\right) ^{m}%
\mathcal{D}_{n}\left( x\right) .  \label{a}
\end{equation}%
For $x=0,$we obtain 
\begin{equation*}
\mathcal{D}_{n+mp^{s}}\equiv \left( -1\right) ^{m}\mathcal{D}_{n}  \label{f}
\end{equation*}%

\end{corollary}
\begin{proof}
The congruence (\ref{a}) follows by setting $f\left( x\right) =x^{n}$ and $t=-1$ in Theorem \ref{T}
\end{proof}
\begin{corollary}
For any prime number $p\geq 3$ and any integers $s\geq 1,m_{0},\cdots
,m_{s}\in \{0,\ldots ,p-1\},$ there holds%
\begin{equation*}
\mathcal{D}_{m_{0}+m_{1}p+\cdots +m_{s}p^{s}}\left( x\right) \equiv \left(
x^{p}-1\right) ^{m_{1}}\left( x^{p^{2}}-1\right) ^{m_{2}}\cdots \left(
x^{p^{s}}-1\right) ^{m_{s}}\mathcal{D}_{m_{0}}\left( x\right) .  \label{s}
\end{equation*}%
In particular, we have%
\begin{align*}
\mathcal{D}_{m_{1}p+\cdots +m_{s}p^{s}}\left( x\right) & \equiv \left(
x^{p}-1\right) ^{m_{1}}\left( x^{p^{2}}-1\right) ^{m_{2}}\cdots \left(
x^{p^{s}}-1\right) ^{m_{s}},   \\
\mathcal{D}_{m_{1}p+\cdots +m_{s}p^{s}}\left( k\right) & \equiv \left(
k-1\right) ^{m_{1}+m_{2}+\cdots ++m_{s}}.  
\end{align*}
\end{corollary}




\end{document}